\documentclass[12pt,oneside]{amsart}

\usepackage{latexsym,amssymb,amsmath,graphicx,hyperref}
\usepackage{verbatim} 

\numberwithin{equation}{section}
\usepackage{tikz}
\usetikzlibrary{shapes}
\usetikzlibrary{arrows}
\tikzstyle{place}=[circle,draw=black!100,fill=black!100,thick,inner sep=0pt,minimum size=1mm]
\tikzstyle{left}=[>=latex,<-,semithick]
\tikzstyle{right}=[>=latex,->,semithick]
\tikzstyle{nleft}=[>=latex,-,semithick]
\tikzstyle{nright}=[>=latex,-,semithick]
\tikzstyle{right2}=[-,semithick]

\newtheorem{theorem}{Theorem}[section]
\newtheorem{lemma}[theorem]{Lemma}
\newtheorem{corollary}[theorem]{Corollary}
\newtheorem{conjecture}[theorem]{Conjecture}

\theoremstyle{definition}
\newtheorem{definition}[theorem]{Definition}
\newtheorem{remark}[theorem]{Remark}
\newtheorem{example}[theorem]{Example}

\newtheorem{observation}[theorem]{Observation}

\begin{document}

\title[Shedding vertices 
of vertex decomposable well-covered 
graphs]{Shedding vertices of vertex decomposable 
well-covered graphs}
\thanks{Last updated: March 5, 2018}

\author{Jonathan Baker}
\address{Brock University, Hamilton, ON, Canada L8K 1V7}
\email{jb16ok@brocku.ca}

\author{Kevin N. Vander Meulen}
\address{Department of Mathematics\\
Redeemer University College, Ancaster, 
ON, Canada L9K 1J4}
\email{kvanderm@redeemer.ca}

\author{Adam Van Tuyl}
\address{Department of Mathematics and Statistics\\
McMaster University, Hamilton, 
ON, Canada L8S 4L8}
\email{vantuyl@math.mcmaster.ca}


\begin{abstract}
We focus our attention on well-covered graphs that are vertex
decomposable. We show that for many
known families of these vertex decomposable graphs, 
the set of shedding vertices forms a dominating set.
We then construct three new infinite families of well-covered
graphs, none of which have this property. 
We use these results to provide a minimal counterexample to
a conjecture of Villarreal regarding Cohen-Macaulay graphs.
\end{abstract}

\maketitle


\section{Introduction}

In this paper we focus on well-covered graphs $G$ that have the additional
property of being vertex decomposable 
(see Definition \ref{graphversion}).
A subset $D$ of the vertex set $V$ of $G$ is a 
{\it dominating set} if every vertex $x \in V\setminus D$
is adjacent to a vertex of $D$. We observe that for most of the
known constructions of pure vertex decomposable graphs, the 
set of shedding vertices ${\rm Shed}(G)$ is a dominating set.
The next result summarizes some of our findings.

\begin{theorem}\label{summary}  Suppose that $G$ is a pure vertex decomposable graph. 
If $G$ is 
\begin{enumerate}
\item[$(i)$] 
a bipartite graph, or
\item[$(ii)$] 
a chordal graph, or
\item[$(iii)$] 
a very well-covered graph, or
\item[$(iv)$] 
a vertex-transitive graph, or
\item[$(v)$] 
a Cameron-Walker graph, or
\item[$(vi)$] 
a clique-whiskered graph, or
\item[$(vii)$] 
a graph with girth at least five,
\end{enumerate}
then ${\rm Shed}(G)$ is a dominating set.
\end{theorem}

\noindent
In particular, $(i)$ is Corollary~\ref{bipartite}, $(ii)$ is 
Theorem~\ref{chordal}, $(iii)$ is Theorem~\ref{verywellcovered},
$(iv)$ is Theorem \ref{circulant}, 
$(v)$ is Corollary \ref{CWgraph},
$(vi)$  is Theorem \ref{cliquewhisker}, and 
$(vii)$ Theorem \ref{girthfive}.

The fact that ${\rm Shed}(G)$ is a dominating set for all these known
vertex decomposable graphs led us to question if this is a feature
of all pure vertex decomposable graphs. Pursuing that question eventually led us
to develop three new infinite families of (vertex decomposable) well-covered graphs.
These infinite families fail to have the property that ${\rm Shed}(G)$ is 
a dominating set and, as we show at the end of the paper,  
provide new counterexamples and insight to a conjecture of Villarreal. 

We outline the structure of this paper.  Section~\ref{VD} 
introduces the definition of pure vertex decomposable graphs
and Section~\ref{shed} describes the set of shedding 
vertices with some introductory tools for identifying them.
Section~\ref{circ} develops our results 
for the chordal and vertex-transitive pure vertex decomposable graphs.
In Section~\ref{vdcons},
we consider two constructions of pure vertex decomposable graphs,
and show that any pure vertex decomposable graph $G$ constructed via 
either construction 
satisfies the property that ${\rm Shed}(G)$ is a dominating set.
In Section~\ref{vwc}, we consider all the very well-covered graphs that
are vertex decomposable.  
In Section~\ref{girth5}, we focus on all pure vertex decomposable graphs with girth
at least five.   
In Section~\ref{newgraphs}, we describe three infinite families of graphs where each 
graph $G$ is pure vertex decomposable, but ${\rm Shed}(G)$ is not
 a dominating set. In Section~\ref{extendgraph}, we show how to take a graph 
 $G$ which is pure vertex decomposable but ${\rm Shed}(G)$ is not a dominating set 
 and duplicate a vertex to construct a larger graph with the same properties.
 We conclude with Section~\ref{insight}, describing how our 
results
 provide new counterexamples for a conjecture of
 Villarreal.
Via a computer search, we find
the smallest pure vertex decomposable graph $G$ for which ${\rm Shed}(G)$
is not a dominating set.  
As part of our computer search, we also show that the set of pure vertex
decomposable graphs is the same as the set of Cohen-Macaulay graphs
for all the graphs on $10$ vertices or fewer. The fact that a minimal counterexample
requires at least nine vertices and that the standard constructions, as
described in Theorem~\ref{summary}, do not provide any counterexamples, make the new constructions
in Section~\ref{newgraphs} relevant for any further analysis of the relationship
between dominating sets and vertex decomposability. 


\section{Vertex Decomposable Graphs}\label{VD}

Let $G$ be a finite simple graph with  
vertex set  $V = \{ x_{1}, \ldots ,x_{n} \}$ and edge set $E$.
We may sometimes write $V(G)$, respectively $E(G)$, for
$V$, respectively $E$, if we wish to highlight that we
are discussing the vertices, respectively edges, of $G$.
A subset $W \subseteq V$ is an \textit{independent} set if no two vertices of 
$W$ are adjacent.  An independent set $W$ is a {\it maximal independent set}
if there is no independent set $U$ such
that $W$ is a proper subset of $U$.  If $W
\subseteq V$ is an independent set, then $V \setminus W$
is a {\it a vertex cover}.  
A vertex cover $C$ is a {\it minimal vertex cover} if $V \setminus C$ is a maximal
independent set.
A graph is {\it well-covered}
if all the maximal independent sets have the same cardinality,
or equivalently, if every minimal vertex cover has the same cardinality.
For example, if $P_n$ is the path graph on $n\geq 2$ vertices, then $P_n$ is well-covered
 if and only if $n=2$ or $n=4$. The graphs in Figure~\ref{well} are 
   well-covered graphs. 

\begin{figure}[ht]
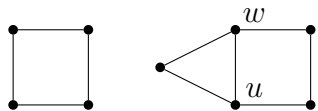


\tikzpicture
\node (1) at (-0.5,0.5)[place] {};
\node (2) at (0.5,0.5)[place] {};
\node (3) at (0.5,-0.5)[place] {};
\node (4) at (-0.5,-0.5)[place] {};
\draw  (2) to (1);
\draw  (3) to (2);
\draw  (4) to (3);
\draw  (4) to (1);
\endtikzpicture \qquad
\tikzpicture
\node (1) at (-0.5,0.5)[place] {};
\node (2) at (0.5,0.5)[place] {};
\node (3) at (0.5,-0.5)[place] {};
\node (4) at (-0.5,-0.5)[place] {};
\node (5) at (-1.5,0)[place]{};
\node (6) at (-0.25, -0.3) {$u$};
\node (7) at (-0.25,0.7) {$w$};
\draw  (2) to (1);
\draw  (3) to (2);
\draw  (4) to (3);
\draw  (4) to (1);
\draw  (5) to (4);
\draw  (1) to (5);
\endtikzpicture

\caption{Two well-covered graphs.}\label{well}
\end{figure}

For any $x \in V$, let $G \setminus x$ denote
the graph $G$ with the vertex $x$ and incident edges removed.
The collection of \textit{neighbours} of a vertex $x \in V$ in $G$,
is the set $N(x) = \{y ~|~ \{x,y\} \in E\}$.   
The \textit{closed neighbourhood} of a vertex $x$ is 
$N[x] = N(x) \cup \{ x \}$.  We 
sometimes write $N_G(x)$ or $N_G[x]$ to highlight which graph
$G$ we are considering.
For $S\subseteq V$, we let $G \setminus S$ denote
the graph obtained by removing all the vertices of $S$ and their incident
edges.  

\begin{definition}\label{graphversion}
A graph $G$ is {\it pure vertex decomposable} if  $G$ is well-covered and 
\begin{enumerate}
\item[$(i)$] $G$ consists of isolated vertices, or $G$ is empty, or
\item[$(ii)$] there exists a vertex $x \in V$, called 
a {\it shedding vertex}, such that
$G \setminus x$ and $G \setminus N[x]$ are pure vertex decomposable.
\end{enumerate}
\end{definition}

For example, the first graph, $C_4$, in Figure~\ref{well} is not pure vertex decomposable
since the deletion on any vertex gives the path $P_3$ which is not well-covered and hence
not pure vertex decomposable. The second graph $G$ in Figure~\ref{well} is pure vertex decomposable: 
$G\setminus u$ is the pure vertex decomposable graph $P_4$ and
$G\setminus N[u]$ is an isolated vertex. 

If $G$ is pure vertex decomposable, then the set of shedding vertices 
is denoted by:
\begin{align*}
{\rm Shed}(G) &= \{x \in V ~|~ \mbox{$G \setminus x$ and $G \setminus N[x]$ are pure 
vertex decomposable}\}.
\end{align*}
For example, ${\rm  Shed}(G)=\{u,w\}$  for the pure vertex decomposable graph in Figure~\ref{well}.

\begin{remark}
The study of vertex decomposable graphs 
lies in the intersection of combinatorial algebraic topology and
combinatorial commutative algebra.
In particular, 
Dochtermann-Engstr\"om \cite{DE} and Woodroofe \cite{W} 
independently showed that vertex decomposability of an independence
complex is a useful tool for exploring algebraic properties of an edge ideal
of a graph. 
The \emph{independence complex} of a graph $G$, denoted
${\rm Ind}(G)$, is the simplicial complex 
\[{\rm Ind}(G) = \{ W \subseteq V ~|~ \mbox{$W$ is an independent set}\}.\]
Vertex decomposability was first
introduced by Provan and Billera \cite{PB} for simplicial complexes.
Our definition of pure vertex decomposability is equivalent to
the statement that the {independence complex} of a graph $G$
is a pure vertex decomposable simplicial complex.  
One can use \cite[Lemma 2.4]{DE} to show the equivalence of definitions.
Provan and Billera's definition required that the simplicial complex
be pure (which translates in the graph case to the condition that $G$ is well-covered).
\end{remark}

\begin{remark}\label{equivalence}
A non-pure version of vertex decomposability was introduced
by Bj\"orner and Wachs \cite{BW}. A 
graph is simply called vertex decomposable if ${\rm Ind}(G)$
satisfies Bj\"orner-Wachs's definition.  Specifically,
we say $G$ is \emph{vertex decomposable} if 
\begin{enumerate}
\item[$(i)$] $G$ consists of isolated vertices, or $G$ is empty, or
\item[$(ii)$] there exists a vertex $x \in V$ such that
\begin{enumerate}
\item[$(a)$] $G \setminus x$ and $G \setminus N[x]$ are 
vertex decomposable, and 
\item[$(b)$] no independent set of $G \setminus N[x]$ is
a maximal independent set of $G \setminus x$.
\end{enumerate} 
\end{enumerate}

One can show that $G$ is pure vertex decomposable if and only
if $G$ is well-covered and vertex decomposable.   It should be noted
that verifying that these two statements are equivalent is subtle.
The proof in both directions is by induction on the number of
vertices.  To show that $G$ is pure vertex decomposable implies 
that $G$ is well-covered and vertex decomposable, one 
needs to treat the cases that the shedding vertex is either 
connected or an isolated vertex as a separate cases.  For the converse
direction, one needs condition $(ii)-(b)$ to verify that $G \setminus
x$ is a well-covered graph.  
\end{remark}

\begin{example}  Expanding upon the above remark, we point out
that definition of a pure vertex decomposable graph allows for more
vertices to be shedding vertices than the definition of vertex
decomposable since isolated vertices can be shedding 
vertices.  Consider the well-covered graph $G$ in Figure \ref{sheddingex}.
\begin{figure}[ht]
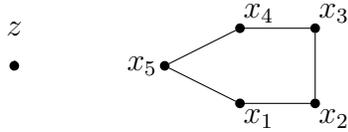

\tikzpicture
\node (8) at (-3.5,0.5) {$z$};
\node (8b) at (-3.5,0)[place] {};
\draw (8b) to (8b);
\node (1) at (-0.5,0.5)[place] {};
\node (2) at (0.5,0.5)[place] {};
\node (3) at (0.5,-0.5)[place] {};
\node (4) at (-0.5,-0.5)[place] {};
\node (5) at (-1.5,0)[place]{};
\node (6) at (-0.25, -.7) {$x_1$};
\node (9) at (.75,-.7) {$x_2$};
\node (11) at (-1.8,0) {$x_5$};
\node (7) at (-0.25,0.7) {$x_4$};
\node (10) at (.75,0.7) {$x_3$};
\draw  (2) to (1);
\draw  (3) to (2);
\draw  (4) to (3);
\draw  (5) to (4);
\draw  (1) to (5);
\endtikzpicture
\caption{A pure vertex decomposable graph.}\label{sheddingex}
\end{figure}
Then the vertex $z$ is a shedding vertex according to the
pure vertex decomposable definition since $G \setminus z = G \setminus N[z]$
is also a pure vertex decomposable graph.  However, $z$ is not
a shedding vertex according to the vertex decomposable
definition since $z$ fails to satisfy condition 
$(ii)-(b)$;  indeed, since $G \setminus N[z] = G \setminus z$,
every maximal independent set of $G \setminus N[z]$ is a maximal
independent set of $G \setminus z$.    Note that $G$ is vertex
decomposable since every vertex $x_i$ for $i=1,\ldots,5$ is a
shedding vertex with respect to Bj\"orner and Wach's definition.
\end{example}

\begin{remark} \label{differentdefn}
We want to highlight that the term shedding vertex appears to
have two different usages in the literature.  In Bj\"orner-Wach's
definition, $x$ is a shedding vertex if it satisfies both
conditions of $(ii)$ given in Remark \ref{equivalence}.  In other papers, e.g. \cite{BC2014} and \cite{W},
a vertex $x$ is a shedding vertex of a graph if it only satisfies
condition $(b)$.   Some care must be taken when applying results
from other papers.  
\end{remark}

The next lemma indicates that when considering vertex decomposable graphs,
it is sufficient to focus on connected graphs.

\begin{lemma}[{\cite[Lemma 20]{W}}]\label{union}
Suppose $G$ and $H$ are disjoint graphs. Then $G \cup H$ is (pure) vertex 
decomposable if and only if $G$ and $H$ are each (pure) vertex decomposable.
\end{lemma} 

By adapting a construction
of 
Biermann, Francisco, H\`a, and Van Tuyl \cite{BFHVT},
we are able to make pure vertex decomposable graphs from any given graph. 
For any graph $G$, let  $S \subseteq V$, and after relabeling,
let $S = \{x_1,\ldots,x_s\}$.  We let $G \cup W(S)$ denote the 
graph with the vertex set $V \cup \{z_1,\ldots,z_s\}$ and edge
set $E \cup \{\{x_i,z_i\} ~|~ i=1,\ldots,s\}$.  The graph
$G \cup W(S)$ is called the {\it whiskered graph at $S$} since
we are adding leaves or ``whiskers'' to all the vertices of $S$.

\begin{theorem}\label{whiskers}\cite[Corollary 4.6]{BFHVT}
Let $G$ be a 
graph and $S \subseteq V$.  If
the induced graph on $V \setminus S$ is a 
chordal graph,
then
$G \cup W(S)$ is vertex decomposable.
In particular, if
the induced graph on $V \setminus S$ is a well-covered 
chordal graph,
then
$G \cup W(S)$ is pure vertex decomposable.
\end{theorem}

\begin{corollary}\label{whiskercor}
If $G$ is any graph with vertex set $V,$ then $G\cup W(V)$ is pure vertex decomposable.
\end{corollary}

\begin{remark}
Corollary \ref{whiskercor} implies that $G \cup W(V)$ is a Cohen-Macaulay
graph (see Section 10).   Villarreal \cite{V1990} was the first
to show that a whiskered graph is a Cohen-Macaulay graph.
\end{remark}


\section{Shedding vertices}\label{shed}

Technically, a vertex 
$x$ is a shedding vertex of a pure vertex decomposable graph $G$
if and only if $G \setminus x$ and $G \setminus N[x]$ 
are both
pure vertex decomposable.  However, as noted in the 
next lemma, to determine if $x$ is a shedding vertex, it is
enough to determine if $G \setminus x$ is a pure  vertex decomposable graph.
The lemma is a direct consequence of known results, such as \cite[Theorem 3.30]{J} or \cite[Proposition 2.3]{PB},
as illustrated in the proof.

\begin{theorem}
\label{VDLink}
Suppose $G$ is pure vertex decomposable. Then $G \setminus N[x]$ is 
pure vertex decomposable for all $x \in V$ and  
${\rm Shed}(G) = \{x \in V ~|~ G \setminus x ~~\mbox{is pure vertex decomposable}
\}.$
\end{theorem}

\begin{proof}
The graph $G \setminus N[x]$ is pure vertex
decomposable if and only if the independence complex ${\rm Ind}(G \setminus N[x])$
is a pure vertex decomposable simplicial complex.  It can be shown that
${\rm Ind}(G \setminus N[x])$ equals the simplicial complex
$${\rm link}(x)=
\{H\subseteq (V\setminus x) ~|~ H\neq\emptyset, H\cup x {\rm{\ is\ an\ independent\ set.}}\},$$ the link of the element $x$ in ${\rm Ind}(G)$.
Then one uses \cite[Proposition 2.3]{PB}, or \cite[Theorem 3.30]{J}, which shows that
every link of a pure vertex decomposable simplicial complex is also
pure vertex decomposable. 
\end{proof}

We now provide some tools that enable us to identify some  elements of ${\rm Shed}(G)$.   
For any $W \subseteq V$,
the {\it induced graph} of $G$ on $W$, denoted $G[W]$, is the graph
with vertex set $W$ and edge set $\{e \in E ~|~ e \subseteq W\}$.
The {\it complete graph} on $n$ vertices, denoted $K_n$,
is the graph on the vertices $\{x_1,\ldots,x_n\}$
with edge set $\{\{x_i,x_j\} ~|~ i \neq j \}$.
A {\it clique} in $G$ is an induced subgraph of $G$ that is
isomorphic to $K_m$ for some $m \geq 1$.

\begin{definition}
A vertex $x \in V$ is a \textit{simplicial vertex} 
if the induced graph on $N(x)$ is a clique; equivalently 
the vertex $x$ appears in exactly one maximal clique of the graph.
A \textit{simplex} is a clique containing at least one simplicial vertex 
of $G$.  A graph $G$ is 
{\it simplicial} if every vertex of 
$G$ is a simplicial vertex or adjacent to one.
\end{definition}

\begin{example}
$(i)$
A vertex $x$ is a {\it leaf} if it has degree one.  Since a leaf
has exactly one neighbour, which is a $K_1$, it is a simplicial vertex.

\noindent
$(ii)$
The graph in Figure \ref{SimplicialGraph} is simplicial. 
The simplicial vertices are $x_1,x_2,x_3$ and $x_4$, and 
each vertex is either a simplicial vertex or adjacent to one.
\begin{figure}[ht]
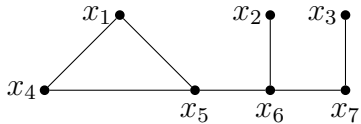

\tikzpicture
\node (1) at (-1,1)[place]{};
\node (1b) at (-1.3,1){$x_1$};
\node (2) at (1,1)[place] {};
\node (2b) at (0.7,1){$x_2$};
\node (3) at (2,1)[place] {};
\node (3b) at (1.7, 1){$x_3$};
\node (4) at (-2,0)[place] {};
\node (4b) at (-2.3, 0){$x_4$};
\node (5) at (0,0)[place]{};
\node (5b) at (0,-0.3){$x_5$};
\node (6) at (1.0,0)[place]{};
\node (6b) at (1, -0.3){$x_6$};
\node (7) at (2,0)[place]{};
\node (7b) at (2, -0.3){$x_7$};
\draw  (1) to (4);
\draw  (1) to (5);
\draw  (2) to (6);
\draw  (3) to (7);
\draw  (4) to (5);
\draw  (5) to (6);
\draw  (6) to (7);
\endtikzpicture 
\caption{A simplicial graph}\label{SimplicialGraph}
\end{figure}
\end{example}

\begin{lemma}\label{wellcoveredsubgraph}
Suppose $G$ is well-covered.  If $x$ is a simplicial vertex,
then for every $y \in N(x)$, the graph $G \setminus y$ is
also well-covered.
\end{lemma}

\begin{proof}
Let $H$ be a maximal independent set of $G \setminus y$.    Then $H$ is also an independent set
of $G$.  If $H$ was not maximal in $G$, then $H \cup \{y\}$ must still be independent in $G$.
This implies $(N[x]\setminus \{y\}) \cap  H = \emptyset$.  
But then $H \cup \{x\}$ would be an independent set of $G \setminus y$, contradicting
the maximality of $H$.  So $H$ is also a maximal independent set of $G$, and since
$G$ is well-covered, all the maximal independent sets of $G \setminus y$ have
the same cardinality.
\end{proof}

\begin{lemma}
\label{SimplicialShedding}
Let $G$ be a pure vertex decomposable graph. If $x$ is a simplicial
vertex, then $N(x) \subseteq {\rm Shed}(G)$.
\end{lemma}

\begin{proof}
Let $N(x) = \{y_1,\ldots,y_d\}$ where $d \geq 1$.  
By Theorem~\ref{VDLink}], it is enough to show that 
$G \setminus y_1$ is pure vertex decomposable.

Let $A_0 = G$, and for $i=1,\ldots,d$, we then
define
\[A_i = A_{i-1} \setminus y_i ~~\mbox{and} ~~B_i = A_{i-1} \setminus N_{A_{i-1}}[y_i].\]
Our goal is to show that $A_1 = A_0 \setminus y_1 = G \setminus y_1$
is a pure vertex decomposable graph.

We first note that for each $i=1,\ldots,d$,
\begin{eqnarray*}
B_i &= &A_{i-1} \setminus N_{A_{i-1}}[y_i] = ((((G \setminus y_1) \setminus y_2) \cdots ) \setminus y_{i-1}) \setminus N_{A_{i-1}}[y_i] \\
& = & G \setminus N_G[y_i]
\end{eqnarray*}
because $\{y_1,\ldots,y_{i-1}\} \subseteq N_G[y_i]$ for each $i=1,\ldots,d$.
It follows from Theorem \ref{VDLink} that each graph $B_i$
is pure vertex decomposable.

Note that $x$ is a simplicial vertex for each graph $A_i$ for $i=0,
\ldots,d-1$.  Since $y_i \in N_{A_{i-1}}(x)$ for $i=1,\ldots,d$, it follows
by repeated use of Lemma \ref{wellcoveredsubgraph} that
each graph $A_i = A_{i-1} \setminus y_i$ for $i=1,\ldots,d$
is a well-covered graph.  
Next, we note that $A_d = ((((G \setminus y_1) \setminus y_2) \cdots ) \setminus y_{d}) = (G \setminus N_G[x]) \cup \{x\}$.    Theorem \ref{VDLink} implies
that $(G \setminus N_G[x])$ is pure vertex decomposable, 
and since an isolated vertex is pure vertex decomposable, 
Lemma \ref{union} implies that $A_d$ is vertex
decomposable.  

Because $A_d = A_{d-1} \setminus y_d$ and $B_d = A_{d-1} \setminus N_{A_{d-1}}[y_d]$ 
are pure vertex decomposable, then by definition, $A_{d-1}$ is pure vertex decomposable.
But then because $A_{d-1}$ and $B_{d-1}$ are pure vertex decomposable, then
so is $A_{d-2}$, and so on.  In particular, $A_1 = G \setminus y_1$
is pure vertex decomposable, as desired.
\end{proof}

\begin{remark}
	It is easier to find examples of non-pure vertex decomposable graphs for which
	the shedding vertices do not constitute a dominating set than it is for pure vertex decomposable graphs. 
	For example, the graph $G$ on five vertices 
	and five edges consisting of a pendant leaf vertex attached to $C_4$  
	is not
	well-covered, but $G$ is vertex decomposable in the non-pure sense. Further,  
	the vertex adjacent to the pendant leaf is the only shedding vertex but this is not a dominating
	set. 
\end{remark}


\section{Vertex-transitive and chordal graphs} \label{circ} 

In this section, we show that the set of shedding vertices
for vertex-transitive graphs and chordal graphs is a dominating set.
A graph $G$ is a {\it vertex-transitive graph} if for every
$x_1,x_2 \in V$ there is a graph automorphism $f:V \rightarrow V$
such that $f(x_1) = x_2$.  We then have the following result.

\begin{theorem}\label{circulant}
Suppose $G$ is a vertex-transitive graph.  If $G$ is pure vertex decomposable,
then ${\rm Shed}(G)$ is a dominating set.
\end{theorem}

\begin{proof}
If $G$ is pure vertex decomposable, then there 
exists some vertex $i$ such that $G \setminus i$ is
pure vertex decomposable. 
By the symmetry of a vertex-transitive graph $G$,  $G \setminus j$ is isomorphic to $G \setminus i$ 
for all $i \neq j$.
 But then ${\rm Shed}(G) = V$, and hence ${\rm Shed}(G)$ is 
 a dominating set.
\end{proof}

A \textit{chordal graph} is a graph $G$ such that every induced
cycle in $G$ has length three.  We have the following classification
of pure vertex decomposable chordal graphs.
 
 \begin{theorem}
 \label{WellcoveredChordal}
 Let $G$ be a chordal graph.  Then the following are equivalent:
 \begin{enumerate}
 \item[$(i)$] $G$ is pure vertex decomposable;
 \item[$(ii)$] $G$ is well-covered;
 \item[$(iii)$] Every vertex of $G$ belongs to exactly one simplex of $G$.
 \end{enumerate}
 \end{theorem}
 
 \begin{proof}
%
 $((ii) \Leftrightarrow (i))$ Woodroofe (\cite[Corollary 7]{W}) (and
independently, Dochtermann and Engstr\"om \cite{DE}) showed
 that every chordal graph is also vertex decomposable.
Now use Remark \ref{equivalence}.
 
 $((ii) \Leftrightarrow (iii))$ This is \cite[Theorem 2]{PTV}.
 \end{proof}

We can now prove the following result.

\begin{theorem}\label{chordal}
Suppose $G$ is a chordal graph.  If $G$ is pure vertex 
decomposable, 
then ${\rm Shed}(G)$ is a dominating set.
\end{theorem}

\begin{proof}
Since $G$ is pure vertex decomposable,
by Theorem \ref{WellcoveredChordal},
every vertex of $G$ belongs to exactly one simplex of $G$.
Thus every vertex is either a simplicial vertex or
adjacent to a simplicial vertex. By Lemma~\ref{SimplicialShedding},
each vertex adjacent to a simplicial vertex is a shedding
vertex. Hence ${\rm Shed}(G)$ is a dominating set.
\end{proof}


\section{Vertex Decomposable Constructions}\label{vdcons}

Given a graph $G$, there are some known constructions (see
\cite{CN,HHKO}) that enable one
to build a new pure vertex decomposable graph 
that contains $G$ as an induced subgraph. 
In this section, we show that the resulting graph for the 
corresponding construction in \cite{CN} and \cite{HHKO} 
has the property
that its set of shedding vertices is a dominating set.

\subsection{Appending cliques}
We first consider a construction of
Hibi, Higashitani, Kimura, and O'Keefe \cite{HHKO} that
builds a pure vertex decomposable graph by appending a clique at each vertex.
More precisely, let $G$ be a  
graph with vertex set
$V(G) = \{x_1,\ldots,x_n\}$ and edge set $E(G)$.
Let $k_1,\ldots,k_n$ be $n$ positive integers with $k_i \geq 2$ for
$i=1,\ldots,n$.  We now construct a graph $\widetilde{G}
= (V(\widetilde{G}),E(\widetilde{G}))$ with
\[V(\widetilde{G}) = 
\{x_{1,1},x_{1,2},\ldots,x_{1,k_1}\} \cup \{x_{2,1},\ldots,x_{2,k_2}\} 
\cup \cdots \cup \{x_{n,1},
\ldots,x_{n,k_n}\}\] 
and edge set
\[E(\widetilde{G}) = 
\left.\left\{\{x_{i,1},x_{j,1}\} ~\right|~ \{x_i,x_j\} \in E(G) \right\} 
\cup  \bigcup_{i=1}^n \left.\left\{ \{x_{i,j},x_{i,l}\} ~\right|~
1 \leq j < l \leq k_i \right\}.\]
That is,
$\widetilde{G}$ is the graph obtained from $G$ by attaching a clique
of size $k_i$ at the vertex $x_i$.  

Starting from any graph $G$, 
the graph $\widetilde{G}$ will always be a pure 
vertex decomposable graph by \cite[Theorem 1]{HHKO}.  
Moreover, the shedding set of any graph $G$  arising from this construction
is a dominating set.

\begin{theorem}
\label{CWVertex}
Given any graph $G$, 
the pure vertex decomposable
graph $\widetilde{G}$ has the property that
${\rm Shed}(\widetilde{G})$ is a dominating set.
\end{theorem}

\begin{proof}
For any $i \in \{1,\ldots,n\}$,
$x_{i,k_i} \neq x_{i,1}$ because
$k_i \geq 2$.  The vertex $x_{i,k_i}$ is a simplicial vertex,
so by Lemma \ref{SimplicialShedding} 
$x_{i,1} \in N(x_{i,k_i}) \subseteq 
{\rm Shed}(\widetilde{G})$.  
Thus $T = \{x_{1,1},\ldots,x_{n,1}\}
\subseteq {\rm Shed}(\widetilde{G})$, and 
$T$ 
is a dominating set of $\widetilde{G}$.
\end{proof}

Hibi et al. \cite{HHKO} developed the above 
construction to study Cameron-Walker graphs.  
A graph $G$ is a \textit{Cameron-Walker} graph if the 
induced matching number $G$ equals the matching number of $G$ 
(see \cite{HHKO} for precise definitions).  One
of the main results of \cite{HHKO}
is the fact that
a Cameron-Walker graph $G$ is a pure vertex
decomposable graph if and only if $G = \widetilde{H}$ for some
graph $H$ (with some hypotheses on the $k_i$'s that appear 
in the construction of $\widetilde{H}$).  Consequently,
we can immediately deduce the following corollary.

\begin{corollary}\label{CWgraph}
Suppose $G$ is a Cameron-Walker graph.  If $G$ is pure vertex 
decomposable, 
then ${\rm Shed}(G)$ is a dominating set.
\end{corollary}


\subsection{Clique-whiskering}
A second construction of pure vertex decomposable graphs 
is due to Cook and Nagel \cite{CN}.  
Let $G$ be a graph on the vertex set $V = \{x_1,\ldots,x_n\}$.
A \textit{clique vertex partition} of $V$ 
is a set $\pi = \{ W_{1},\ldots ,W_{t} \}$ of disjoint 
 subsets that partition $V$ such that each induced graph
$G[W_i]$ is a clique.
A \textit{clique-whiskered} graph $G^{\pi}$ constructed from
the graph $G$ with clique partition $\pi = \{W_1,\ldots,W_t\}$
is the graph with $V(G^\pi) = \{x_1,\ldots,x_n,w_1,\ldots,w_t\}$
and
$E(G^\pi) = E \cup \{ \{ x,w_{i} \} ~|~ x \in W_{i} \}$.
In other words, for each clique in the partition $\pi$,
we add a new vertex $w_i$, and join $w_i$ to all the vertices in the clique.

Note that if $\widetilde{G}$ is the graph obtained from $G$ by  
appending cliques with $k_1 = \cdots = k_n = 2$, 
then $\widetilde{G}$ is isomorphic to the clique-whiskered graph 
$G^\pi$ using the clique partition  $\pi = \{\{x_1\},\{x_2\},\ldots,\{x_n\}\}$.

Cook and Nagel (\cite[Theorem 3.3]{CN}) showed that
for any graph $G$ and any clique partition
$\pi$ of $G$, the graph $G^\pi$ is always pure vertex decomposable.
Like the previous construction, any graph constructed
via this method has ${\rm Shed}(G)$ as a dominating set.

\begin{theorem}\label{cliquewhisker}
\label{CliqueWhiskerShedding}
Let $G$ be a  
graph with clique
partition $\pi$.
The pure vertex decomposable graph $G^\pi$ 
has the property that ${\rm Shed}(G^\pi)$ is a dominating
set.
\end{theorem}

\begin{proof}
If $\pi = \{W_1,\ldots,W_t\}$, then
the vertex set of $G^\pi$ is $\{x_1,\ldots,x_n,w_1,\ldots,w_t\}$.   
Every vertex $x_i$ belongs to some clique
$W_j$.  So, in $G^\pi$, the vertex $x_i$ is adjacent to
$w_j$.  By construction, $w_j$ is adjacent only to
the vertices of $W_j$, and since $W_j$ is a clique,
$w_j$ is a simplicial vertex.  
Thus by Lemma \ref{SimplicialShedding}, $x_i 
\in N(w_j) \subseteq {\rm Shed}
(G^\pi)$.  Thus $\{x_1,\ldots,x_n\} \subseteq {\rm Shed}(G^\pi)$,
and this subset forms a dominating set.
\end{proof}


\section{Very well-covered graphs}\label{vwc}

A well-covered graph is {\it very well-covered} if 
every maximal independent set
has cardinality $|V|/2$.  
Very well-covered graphs are known~\cite{F} to have a perfect
matching with a 
neighbour connectedness property. 
A {\it matching} is a subset of edges
of $G$ that do not share any common endpoints.
A matching is {\it perfect} if the set of
vertices in the edges of the matching are
all of the vertices
Given $M$ is a perfect matching
of $G$, we let $M(x)$ denote the vertex matched with $x$.  The matching
$M$ has the \emph{neighbour
connectedness property} if for every vertex $x$ of $G$, if $y\in N(x)$ and $y\neq M(x)$
then $y\not\in N(M(x))$ and $y\in N(z)$ for every $z\in N(M(x))$.

\begin{lemma}\cite[Theorem 1.2]{F}\label{propertyP}
	A graph $G$ is very well-connected if and only if $G$ has at least one perfect matching $M$ and every perfect 
	matching has the neighbour connectedness property.
\end{lemma}

\begin{theorem}[{\cite[Theorem 3.2]{MMCRTY} and \cite[Theorem 0.2]{CRT}}]\label{perfect}
If $G$ is very well-covered then the following are equivalent:
\begin{enumerate}
	\item $G$ is pure vertex decomposable;
	\item $G$ is Cohen-Macaulay;
	\item $G$ has a unique perfect matching.
\end{enumerate}
\end{theorem} 

If $M$ is a perfect matching in $G$, we say an even cycle $C$ is \emph{$M$-alternating} if 
half of the edges of $C$ are in $M$. In the following argument, we use that fact that if $G$
has an $M$-alternating even cycle, then $G$ does not have a unique perfect matching.

\begin{theorem}\label{verywellcovered}
\label{WellCoveredSheddingSet}
Let $G$ be a very well-covered graph.  If $G$ is pure vertex decomposable,
then ${\rm Shed}(G)$ is a dominating set.
\end{theorem}

\begin{proof}
Suppose $G$ is a very well-covered pure vertex decomposable graph. 
Since $G$ is pure vertex decomposable, $G$ has a unique perfect matching $M$ (by Theorem~\ref{perfect}) that has the 
neighbour connectedness property (Lemma~\ref{propertyP}).

Let 
\[S = \bigcup_{\mbox{$z$ is a leaf of $G$}} N(z).\]
We claim that $S$ is a dominating set. We demonstrate this by 
showing that if $S$ is not dominating, then $G$ has an $M$-alternating even cycle, contradicting the
fact that $G$ has a unique perfect matching. 
 
Suppose there exists a vertex $w$ such that $w\not\in S$ and $w$ is not adjacent to any vertex in $S$. 
In particular, $w$ has a neighbour $x_1$ distinct from $M(w)$. (For convenience, we let $x_0=M(w)$.) 
Now $M(x_1)$ is not a leaf since $w$ is not adjacent
to any vertex in $S$. Thus there exists a vertex $x_2\neq x_1$ adjacent to $M(x_1)$.  If $x_2=x_0$, then $G$ has an $M$-alternating
 four-cycle. Thus assume $x_2\neq x_0$. 
 Note that by the neighbour connectedness
property, $x_2\not\in N(x_1)$, and $x_2 \in N(w)$. Also, by the neighbour connectedness property, 
$M(x_2)$ is not adjacent to $w$ or $M(x_1)$. 

Again, since $w$ is not adjacent to any vertex in $S$, there exists a 
vertex $x_3\neq x_2$ that is adjacent to $M(x_2)$. If $x_3\in\{x_1,x_0\}$ then $G$ has an $M$-alternating four-cycle.   
Thus assume $x_3\not\in\{x_1,x_0\}$. By the neighbour connectedness property, 
$x_3$ is not adjacent to $x_1$ or $x_2$, and $x_3$ is adjacent to $w$ and $M(x_1)$. 
By the neighbour conectedness  property, $M(x_3)$ is not adjacent to any vertex in $\{w,M(x_1),M(x_2)\}$.

Repeating the argument, we can obtain a sequence of vertices $x_1,x_2,x_3,\ldots \in N(w)$ with $x_i$ adjacent
to $M(x_{i-1})$ for all $i>1$. Further,
$w$ is not adjacent to $M(x_1),M(x_2),M(x_3),\cdots$ by the neighbour connectedness property. Thus
$x_i\neq M(x_j)$ for any $i$ and $j$. Since $G$ is finite, there must exist some $i>1$ such that $x_i=x_j$ for some $j$ with $0\leq j<i$.
Thus $G$ has an $M$-alternating even cycle $\{x_j,M(x_j),x_{j+1},M(x_{j+1}),\ldots,x_{i-1},M(x_{i-1}),x_j\}$.

Therefore $S$ is a dominating set. By
Lemma~\ref{SimplicialShedding},  
$S \subseteq {\rm Shed}(G)$. 
Therefore ${\rm Shed}(G)$ is a dominating set.
\end{proof}

\begin{corollary}\label{bipartite}
Suppose $G$ is a bipartite graph.  If $G$ is pure vertex 
decomposable, then ${\rm Shed}(G)$ is a dominating set.
\end{corollary}

\begin{proof}
If $G$ is pure vertex decomposable, then $G$ is well-covered. 
In that case, since $G$ is bipartite, $G$ is very well-covered.
The result now follows from Theorem \ref{WellCoveredSheddingSet}.
\end{proof}

\begin{remark}
As we noted in the previous proof, 
the class of very well-covered graphs
contains the family of well-covered bipartite graphs.
Theorem~\ref{perfect} 
can be viewed as a generalization of
results first proved about well-covered bipartite graphs.
Herzog and Hibi gave a combinatorial classification
of Cohen-Macaulay bipartite graphs in \cite[Corollary 9.1.14]{HH2010}. 
The classification of  very well-covered graphs in Theorem~\ref{perfect}
generalizes Herzog and 
Hibi's work.  Van Tuyl \cite{V}  
showed that a bipartite graph
is well-covered and vertex decomposable if and only if it is Cohen-Macaulay.
\end{remark}


\section{Graphs with girth at least five}\label{girth5}

We now consider all pure vertex decomposable graphs 
with girth five or larger.  Vertex decomposable 
graphs from this class were independently classified by 
B\i y\i ko\u{g}lu and Civan \cite{BC2014} and
Hoang, Minh, and Trung \cite{HMT}.
Both of these results relied on the 
classification of well-covered graphs
with girth five or larger due
to Finbow, Hartnell, and Nowakowski \cite{FHN}.

To state the required classification,
we first review the relevant background.
The \textit{girth} of a graph $G$ is the number of vertices of a smallest 
induced cycle of $G$.  If $G$ has no cycles,
then we say $G$ has infinite girth. 
A \textit{pendant edge} is an edge that is incident to a leaf.

An induced $5$-cycle $B$ is said to be {\it basic}
if no pair of adjacent vertices in $B$ have 
degree three or larger in $G$.
A graph $G$ is in the class $\mathcal{PC}$ if $V$ can be partitioned into subsets 
$V = P \cup C$ where $P$ contains all the vertices incident with pendant edges 
and the pendant edges form a perfect matching of $P$, 
and where $C$  contains the vertices of basic 5-cycles, 
and these basic 5-cycles form a partition of $C$.

We then have the following classification 
(see the cited papers for additional equivalent statements).

\begin{theorem}[\cite{BC2014,HMT}]
\label{VDLargeGirth}
Let $G$ be a connected graph of girth at least 5. If $G$ is well-covered, 
then the following are equivalent:
\begin{enumerate}
\item[$(i)$] $G$ is vertex decomposable; 
\item[$(ii)$] $G$ is either an isolated vertex or in the class $\mathcal{PC}$.
\end{enumerate}
\end{theorem}

We first prove a lemma.

\begin{lemma}
\label{PCShedding}
Let $B$ be a basic $5$-cycle of a well-covered graph $G \in \mathcal{PC}$. 
If $B$ has a vertex $x$ adjacent to two vertices of $B$ of degree two in $G$, then 
$x \in {\rm Shed}(G)$.
\end{lemma}

\begin{proof} 
The statement of the lemma is embedded in \cite{HMT} in their proof of
Theorem \ref{VDLargeGirth} that was stated above.
In particular, \cite[Lemma 2.2]{HMT} (which is used to
prove \cite[Theorem 2.4]{HMT}) shows that if a graph $G$ is in 
$\mathcal{SC}$, a class that contains the graphs of $\mathcal{PC}$,
then $G$ is vertex decomposable.  Moreover, to prove this
fact, the authors show that the vertex $x$ in our statement is the required
shedding vertex.  As an aside, a similar argument is found
in \cite[Lemma 5]{FHN} for extendable vertices.  One could also
use \cite[Lemma 16]{W}, but note that the definition of a shedding vertex
is not the same as our usage; one still needs to
show that $G \setminus x$ and $G \setminus N[x]$ are 
vertex decomposable (see Remark \ref{differentdefn}).
\end{proof}

\begin{theorem}\label{girthfive}
Let $G$ be a graph with girth of at least five.
If $G$ is pure vertex decomposable, then
${\rm Shed}(G)$ is a dominating set.
\end{theorem}

\begin{proof}
If $G$ is vertex decomposable, by Theorem \ref{VDLargeGirth}, $G$ is either a single
vertex or $G \in \mathcal{PC}$.    Because the statement is vacuous for a single vertex,
we can assume that $G \in \mathcal{PC}$. 
Let $V = P \cup C$ be the corresponding partition of $G$ and let $x$ be a vertex of $G$. 

Suppose $x \in P$.  Then $x$ is either a leaf or adjacent to a leaf $y$.  
So by Lemma~\ref{SimplicialShedding}, $x$ is a shedding vertex of $G$ or adjacent to one.

Suppose $x \in C$. Then there is a basic $5$-cycle $B$ such that $x \in V(B)$.
If $x$ is adjacent to two vertices of degree two, then $x \in {\rm Shed}(G)$ by Lemma \ref{PCShedding}.
So suppose that there exists $y \in V(B)$ adjacent to $x$ such that $y$ has degree at least three.
Because $B$ is a basic $5$-cycle, 
$y$ must be adjacent to two vertices of degree two.
By Lemma~\ref{PCShedding}, $y \in {\rm Shed}(G)$.  Hence $x$ is adjacent
to a shedding vertex. Therefore every vertex in $C$ is a shedding vertex of $G$ or
adjacent to one.
\end{proof}


\section{Three new vertex decomposable graphs}\label{newgraphs}

In this section we will construct three infinite family of graphs.
Each family will have the property that all
members are pure vertex decomposable, but ${\rm Shed}(G)$ is {\bf not}
a dominating set. In particular, for each construction, the vertices in 
${\rm Shed}(G)$ are part of a clique of vertices $Z$, none of which
is adjacent to any vertex in a non-empty set $X$.

\subsection{Construction 1} 
Fix $m$ integers $k_i \geq 2$, and suppose that
$k_1 + \cdots + k_m = n$.   We define 
$D_n(k_1,\ldots,k_m)$ to be the graph on the $5n$ vertices
\[V = X \cup Y \cup Z = 
\{ x_1,\ldots,x_{2n} \} \cup \{ y_1,\ldots,y_{2n} \}
\cup \{ z_1,\ldots,z_n \}\]
with the edge set given by the following conditions:
\begin{enumerate}
\item[$(i)$] the induced graph on $Z$ is a complete graph $K_n$;
\item[$(ii)$] $Y$ is an independent set, i.e., $G[Y] = \overline{K_{2n}}$, where
$\overline{H}$ denotes the complement of the graph $H$;
\item[$(iii)$] the induced graph $G[X]$ is $K_{k_1,k_1} \sqcup \cdots \sqcup K_{k_m,k_m}$
where the vertices of $G[X]$ are labeled so that the $i$-th complete bipartite
graph has bipartition 
\[\{ x_{2w+1},x_{2w+3},\ldots, x_{2(w+k_i)-1} \} \cup \{x_{2w+2},x_{2w+4},\ldots, x_{2(w+k_i)}\}\]
 with $w=\sum_{\ell=1}^{i-1} k_\ell$ where $w = 0$ if $i=1$;
\item[$(iv)$] $\{x_j,y_j\}$ are edges for $1 \leq j \leq 2n$; and 
\item[$(v)$] $\{z_j,y_{2j}\}$ and $\{z_j,y_{2j-1}\}$
are edges for $1 \leq j \leq n$.
\end{enumerate}

Roughly speaking, the graph $D_n(k_1,\ldots,k_m)$ is formed by ``joining'' 
$m$ complete bipartite graphs to a complete graph $K_n$ by first passing through
an independent set of vertices $Y$.  
Going forward, it is useful to make the observation that the induced graph $G[X \cup Y]$ has
a perfect matching given by the edges $\{x_j,y_j\}$ for $j = 1,\ldots, 2n$.

\begin{example}
To illustrate our construction, the graph $D_5(2,3)$ is given
in Figure \ref{D5}.

\begin{figure}[ht]
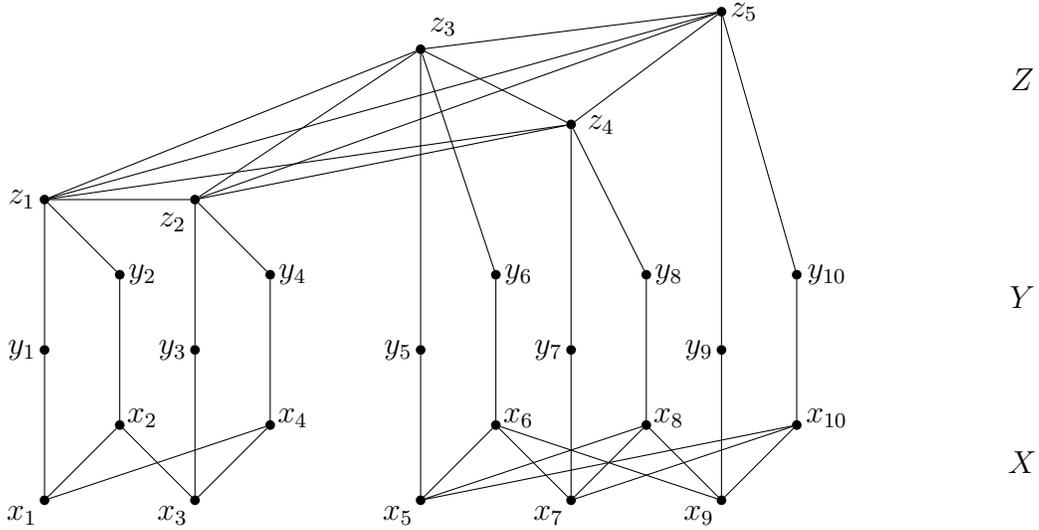

\tikzpicture
\node (x1) at (-5,-2)[place]{};
\node (x1b) at (-5.3, -2.2){$x_1$};
\node (x3) at (-3,-2)[place]{};
\node (x3b) at (-3.3, -2.2){$x_3$};
\node (x5) at (0,-2)[place]{};
\node (x5b) at (-0.3, -2.2){$x_5$};
\node (x7) at (2,-2)[place]{};
\node (x7b) at (1.7, -2.2){$x_7$};
\node (x9) at (4,-2)[place]{};
\node (x9b) at (3.7, -2.2){$x_9$};
\node (x2) at (-4,-1)[place]{};
\node (x2b) at (-3.7, -0.9){$x_2$};
\node (x4) at (-2,-1)[place]{};
\node (x4b) at (-1.7, -0.9){$x_4$};
\node (x6) at (1,-1)[place]{};
\node (x6b) at (1.3, -0.9){$x_6$};
\node (x8) at (3,-1)[place]{};
\node (x8b) at (3.3, -0.9){$x_8$};
\node (x10) at (5,-1)[place]{};
\node (x10b) at (5.4, -0.9){$x_{10}$};
\node (y1) at (-5,0)[place]{};
\node (y1b) at (-5.3, 0){$y_1$};
\node (y3) at (-3,0)[place]{};
\node (y3b) at (-3.3, 0){$y_3$};
\node (y5) at (0,0)[place]{};
\node (y5b) at (-0.3, 0){$y_5$};
\node (y7) at (2,0)[place]{};
\node (y7b) at (1.7, 0){$y_7$};
\node (y9) at (4,0)[place]{};
\node (y9b) at (3.7, 0){$y_9$};
\node (y2) at (-4,1)[place]{};
\node (y2b) at (-3.7, 1){$y_2$};
\node (y4) at (-2,1)[place]{};
\node (y4b) at (-1.7, 1){$y_4$};
\node (y6) at (1,1)[place]{};
\node (y6b) at (1.3, 1){$y_6$};
\node (y8) at (3,1)[place]{};
\node (y8b) at (3.3, 1){$y_8$};
\node (y10) at (5,1)[place]{};
\node (y10b) at (5.4, 1){$y_{10}$};
\node (z1) at (-5,2)[place]{};
\node (z1b) at (-5.3, 2){$z_1$};
\node (z2) at (-3,2)[place]{};
\node (z2b) at (-3.3, 1.7){$z_2$};
\node (z3) at (0,4)[place]{};
\node (z3b) at (0.3, 4.3){$z_3$};
\node (z4) at (2,3)[place]{};
\node (z4b) at (2.4, 3){$z_4$};
\node (z5) at (4,4.5)[place]{};
\node (z5b) at (4.3, 4.5){$z_5$};
\draw (x1) to (x2);
\draw (x1) to (x4);
\draw (x3) to (x2);
\draw (x3) to (x4);
\draw (x5) to (x6);
\draw (x5) to (x8);
\draw (x5) to (x10);
\draw (x7) to (x6);
\draw (x7) to (x8);
\draw (x7) to (x10);
\draw (x9) to (x6);
\draw (x9) to (x8);
\draw (x9) to (x10);
\draw (x1) to (y1);
\draw (x2) to (y2);
\draw (x3) to (y3);
\draw (x4) to (y4);
\draw (x5) to (y5);
\draw (x6) to (y6);
\draw (x7) to (y7);
\draw (x8) to (y8);
\draw (x9) to (y9);
\draw (x10) to (y10);
\draw (y1) to (z1);
\draw (y3) to (z2);
\draw (y5) to (z3);
\draw (y7) to (z4);
\draw (y9) to (z5);
\draw (y2) to (z1);
\draw (y4) to (z2);
\draw (y6) to (z3);
\draw (y8) to (z4);
\draw (y10) to (z5);
\draw (z1) to (z2);
\draw (z1) to (z3);
\draw (z1) to (z4);
\draw (z1) to (z5);
\draw (z2) to (z3);
\draw (z2) to (z4);
\draw (z2) to (z5);
\draw (z3) to (z4);
\draw (z3) to (z5);
\draw (z4) to (z5);
\node (Z) at (8, 3.6){$Z$};
\node (Y) at (8, 0.7){$Y$};
\node (X) at (8, -1.5){$X$};
\endtikzpicture 
\caption{The graph $D_{5}(2,3)$.}\label{D5}
\end{figure}
\end{example}

We now show that the graphs $D_n(k_1,\ldots,k_m)$ are all well-covered.
In what follows, we write $\alpha(G)$ to denote the cardinality of a maximal
independent set in $G$.

\begin{lemma}
\label{ClassWellCovered}
The graph $D_n(k_1,\ldots,k_m)$ is well-covered.
\end{lemma}

\begin{proof}  Let $G=D_{n}(k_1,\ldots,k_m)$.  It suffices to show that
every maximal independent set has the same cardinality.

We can partition $V$ into $n$ sets of five vertices, namely, 
$\{x_{2i-1},x_{2i},y_{2i-1},y_{2i},z_{i} \}$ for $1 \leq i \leq n$.
The induced graph on each such set  
is a five cycle. Since $\alpha(C_5)= 2$, it follows that $\alpha(G)\leq 2n$.
On the other hand, $Y$ is a maximal independent set of vertices with $|Y| = 2n$,
so $\alpha(G) = 2n$.

Let $H$ be any maximal independent set with $|H| < 2n$.
If $H \cap Z = \emptyset$, then because there are $2n$ edges of the form
$\{x_j,y_j\}$, there exists an $i$ such that neither $x_i$ nor $y_i$ belong to
$H$.  But then $H \cup \{y_i\}$ is an independent set  since $y_i$ 
is only adjacent to a vertex in $Z$ and $x_i$.  This contradicts the fact that
$H$ is a maximal independent set.

So, there exists a $z_i \in H \cap Z$.  Because $G[Z]$ is a complete graph, $H \cap Z = \{z_i\}$.
Thus each edge $\{x_{j},y_{j}\}$ for $j \neq 2i$ or $2i-1$ has a vertex in $H$, 
otherwise $H\cup \{y_{j}\}$ is a larger independent set.  Because $|H| \leq 2n-1$,
we have already accounted
for all the vertices in $H$.  
So, neither $x_{2i}$ nor $x_{2i-1}$ are in $H$.
Hence $x_{2i}$, respectively $x_{2i-1}$, is adjacent to some vertex $x_{l}\in H$,
respectively $x_{k} \in H$. Further, $x_{2i-1},x_l,x_{2i},x_k$ all 
belong to the same complete bipartite graph  $K_{k_r,k_r}$.
Then $l$ must be odd since $2i$ is even and $k$ must be even since $2i-1$ is odd.
However, then 
$x_k$ is adjacent to $x_l$, 
contradicting
the fact that $x_k,x_l \in H$.
Thus $H$ cannot be a maximal independent set if $|H| <2n$, and so
every maximal independent set has cardinality $2n$.  Therefore $G$ is well-covered.
\end{proof}

We now show that any graph made via our construction is pure vertex decomposable, and furthermore,
we determine its set of shedding vertices.

\begin{theorem}
\label{ClassDOnlyKN}
If $G = D_n(k_1,\ldots,k_m)$ then $G$ is pure vertex decomposable and ${\rm Shed}(G) = Z$.
\end{theorem}

\begin{proof}
Let $G=D_{n}(k_1,\ldots,k_m)$.   
By Lemma \ref{ClassWellCovered}, $G$ is well-covered.
We show that $G$ is pure vertex decomposable by first working through four 
claims.

\noindent
{\it Claim 1:}  For each $i = 1,\ldots,n$, $G_i = (((G \setminus z_1) \setminus z_2) \cdots \setminus z_i)$
is a well-covered graph.

\noindent
Fix some $i \in \{1,\ldots,n\}$.
Let $H$ be any maximal independent set of $G_i$.  
Since $\{x_1,x_2\},\ldots,$ $\{x_{2i-1},x_{2i}\}$
are  edges of $G_i$, for each $j=1,\ldots,i$, $H$ contains
at most one of $x_{2j-1}$ and $x_{2j}$.  
Then $H$ contains at least one
of $y_{2j-1}$ or $y_{2j}$ for each $j=1,\ldots,i$, since $H$ is maximal and
$y_{2j-1}$ and $y_{2j}$ are leaves in $G_i$. 
But then 
$H$ is also a maximal independent set of $G$ since each vertex $z_1,\ldots,z_i$ of $G$
is adjacent to at least one vertex in $H$. 
Because $G$ is well-covered,
$|H| = \alpha(G)$.  So $G_i$ is also well-covered. 

\noindent
{\it Claim 2:} The graph $G_n$ is pure vertex decomposable.

\noindent
The graph $G_n$ is the same as the induced graph $G[X \cup Y]$.  
So $G_n$ is the graph of $m$ disjoint graphs, where the $j$-th connected component
is the complete bipartite graph $K_{k_j,k_j}$ with whiskers at every vertex.  
Now use
Corollary ~\ref{whiskercor} and Lemma~\ref{union} to finish the proof.

\noindent
{\it Claim 3:} For each $i=1,\ldots,n$, $N_i = G_{i-1} \setminus N[z_i]$ is a well-covered graph.

\noindent
For a fixed $i$, suppose that $x_{2i-1}$ and $x_{2i}$ appear in
the complete bipartite graph $K_{k_j,k_j}$.  Then the graph $N_i$ consists
of $m$ disjoint graphs: $m-1$ of these graphs
are the complete bipartite graphs 
with whiskers at every vertex, and the $m$-th graph is the
graph $K_{k_j,k_j}$ with whiskers at every vertex except $x_{2i-1}$ and $x_{2i}$.
Note that $m-1$ graphs are well-covered as was argued in Claim 2. 
The $m$-th graph is also well-covered:  let $S=V(K_{k_j,k_j}\setminus \{ x_{2i-1},x_{2i} \})$
and apply Theorem~\ref{whiskers} to $K_{k_j,k_j} \cup W(S)$.
Therefore $N_i$ is well-covered.

\noindent
{\it Claim 4:} For each $i=1,\ldots,n$, $N_i$ is pure vertex decomposable.

\noindent
As shown in the previous proof, $N_i$ is made up
of $m$ disjoint graphs, where each graph is either a complete bipartite
graph with whiskers at every vertex, or a complete bipartite graph
with whiskers at every vertex except at two adjacent vertices.  It follows
from Theorem \ref{whiskers} that in both cases, each disjoint
graph is pure vertex decomposable.  By Lemma \ref{union}, it then follows that
$N_i$ is pure vertex decomposable.

Thus we have established Claims 1--4. 
By definition, $G$ is pure vertex decomposable if we can show that $G_1$ and $N_1$ are 
pure vertex decomposable.  But $G_1$ is 
pure vertex decomposable if we can show that $G_2$ and $N_2$ are vertex
decomposable.  Continuing in this fashion, to show that $G$ is pure vertex decomposable,
it suffices to show that $G_n$ and $N_1,\ldots,N_n$ are all pure vertex decomposable.
But this was shown in Claims 1--4.  So $G$ is pure vertex decomposable.

We next observe that ${\rm Shed}(G) = Z$.  Note that to show $G$ is pure vertex decomposable,
we showed that $z_1 \in {\rm Shed}(G)$.  
By graph symmetry, $z_j \in {\rm Shed}(G)$ for any $z_j \in Z$.  So
$Z \subseteq {\rm Shed}(G)$.

Next, we show $Y \cap {\rm Shed}(G) = \emptyset$.  Let $y \in Y$. 
 After relabeling, assume that $y = y_{2n}$.
Then $\{y_1,\ldots,y_{2n-1},x_{2n}\}$ and
$\{z_{1},x_{1}, y_{3},\ldots,y_{2n-2}, x_{2n-1}\}$
are maximal independent sets, in $G\setminus y$, of cardinality $2n$ and $2n-1$ respectively. 
Thus $G \setminus y$ is not well-covered and so $y\not\in {\rm Shed}(G)$.

Finally, we show that $X \cap {\rm Shed}(G) = \emptyset$.  Again, we 
show that for any $x \in X$, the graph $G \setminus x$ is not well-covered.
After relabeling, assume $x = x_1$.
The set $Y$ is an independent set of $G \setminus x$ of cardinality $2n$.
Note that since $k_{1} \geq 2$, the  vertex $x_{3}$  is adjacent to $x_{2}$ and $x_4$.
It follows that $L=\{z_1,x_3, y_{4},\ldots,y_{2n}\}$ is a maximal independent set of $G\setminus x$ 
with $2n-1$ vertices.

Thus ${\rm Shed}(G) = Z$, as desired.
\end{proof}

The graphs constructed in this subsection give us the first family of graphs $G$
for which ${\rm Shed}(G)$ is not a dominating set, 
 since no vertex in $X$ is adjacent to any
vertex in $Z$.

\begin{corollary}
If $G = D_n(k_1,\ldots,k_m)$ then ${\rm Shed}(G)$ is not a dominating set.
\end{corollary}


\subsection{Construction 2}

Next we construct a graph $G=F_{m}$ with vertex set $V=X\cup Y\cup Z$ with $X= \{ x_1,\ldots,x_{2m} \}$, 
$Y= \{ y_1,y_2\}$, and $Z= \{ z_1,z_2,z_3 \}$
and  edge set given by the following conditions:
\begin{enumerate}
\item[$(i)$] the induced subgraph $G[X]$  
is the $m$-partite graph $K_{2,2,\ldots,2}$, whose
  complement is the matching with edges $\{x_{2i-1},x_{2i}\}$, $1\leq i \leq m$;
\item[$(ii)$] $y_1$ is adjacent to $z_1$ and each $x_{2i-1}$, $1\leq i \leq m$; 
\item[$(iii)$] $y_2$ is adjacent to $z_2$ and each $x_{2i}$ for $1\leq i \leq m$; and
\item[$(iv)$] the induced subgraph on $Z$ is $K_{3}$.
\end{enumerate}

Note that if we let $X_1 =\{x_1,x_3,\ldots, x_{2m-1}\} \cup \{ y_1\}$ and
$X_2 =\{x_2,x_4,\ldots, x_{2m}\} \cup \{ y_2\}$, then 
$G[X_1]$ and $G[X_2]$ are both cliques isomorphic to $K_{m+1}$. 

\begin{example}
Two examples of Construction 2 are drawn below.  In particular,
the graph $F_2$ is in Figure \ref{PD2}, and the graph $F_3$
is drawn in Figure \ref{PD3}.
\begin{figure}[ht]
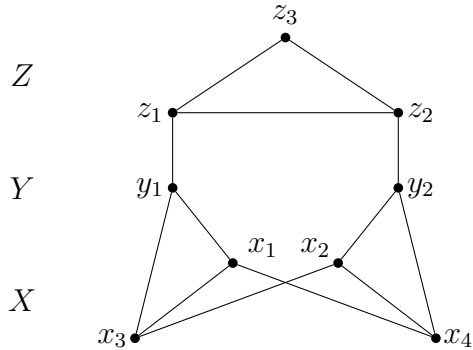

\tikzpicture
\node (1) at (0,3)[place]{};
\node (1b) at (0,3.3){$z_3$};
\node (2) at (-1.5,2)[place] {};
\node (2b) at (-1.8, 2){$z_1$};
\node (3) at (1.5,2)[place] {};
\node (3b) at (1.8,2){$z_2$};
\node (4) at (-1.5,1)[place] {};
\node (4b) at (-1.8, 1){$y_1$};
\node (5) at (1.5,1)[place]{};
\node (5b) at (1.8,1){$y_2$};
\node (6) at (-0.7,0)[place]{};
\node (6b) at (-0.3, 0.2){$x_1$};
\node (7) at (-2,-1)[place]{};
\node (7b) at (-2.3, -1){$x_3$};
\node (8) at (2,-1)[place]{};
\node (8b) at (2.3, -1){$x_4$};
\node (9) at (0.7,0)[place]{};
\node (9b) at (0.4, 0.2){$x_2$};
\node (10b) at (-3.5, 2.5){$Z$};
\node (10b) at (-3.5, 1){$Y$};
\node (10b) at (-3.5, -0.5){$X$};
\draw  (1) to (2);
\draw  (2) to (3);
\draw  (1) to (3);
\draw  (2) to (4);
\draw  (3) to (5);
\draw  (4) to (6);
\draw  (4) to (7);
\draw  (5) to (8);
\draw  (5) to (9);
\draw  (6) to (7);
\draw  (6) to (8);
\draw  (8) to (9);
\draw  (9) to (7);
\endtikzpicture 
\caption{The graph $F_{2}$.}\label{PD2}
\end{figure}

\begin{figure}[ht]
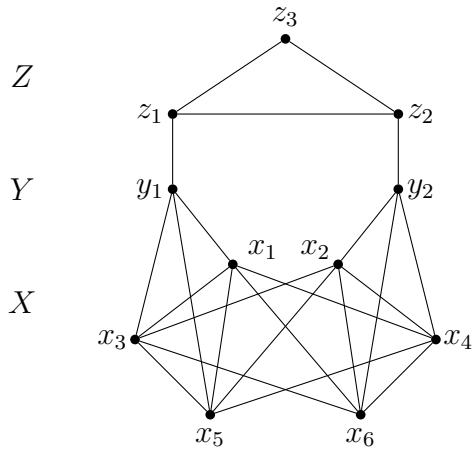

\tikzpicture
\node (1) at (0,3)[place]{};
\node (1b) at (0,3.3){$z_3$};
\node (2) at (-1.5,2)[place] {};
\node (2b) at (-1.8, 2){$z_1$};
\node (3) at (1.5,2)[place] {};
\node (3b) at (1.8,2){$z_2$};
\node (4) at (-1.5,1)[place] {};
\node (4b) at (-1.8, 1){$y_1$};
\node (5) at (1.5,1)[place]{};
\node (5b) at (1.8,1){$y_2$};
\node (6) at (-0.7,0)[place]{};
\node (6b) at (-0.3, 0.2){$x_1$};
\node (x5) at (-1,-2)[place]{};
\node (x5b) at (-1,-2.3){$x_5$};
\node (x6) at (1, -2)[place]{};
\node (x6b) at (1,-2.3){$x_6$};
\draw (x5) to (4);
\draw (x5) to (6);

\draw (x6) to (5);
\draw (x6) to (6);

\node (7) at (-2,-1)[place]{};
\node (7b) at (-2.3, -1){$x_3$};
\node (8) at (2,-1)[place]{};
\node (8b) at (2.3, -1){$x_4$};
\node (9) at (0.7,0)[place]{};
\node (9b) at (0.4, 0.2){$x_2$};

\draw (x5) to (7);
\draw (x5) to (9);
\draw (x5) to (8);

\draw (x6) to (7);
\draw (x6) to (9);
\draw (x6) to (8);

\node (10b) at (-3.5, 2.5){$Z$};
\node (10b) at (-3.5, 1){$Y$};
\node (10b) at (-3.5, -0.5){$X$};

\draw  (1) to (2);
\draw  (2) to (3);
\draw  (1) to (3);
\draw  (2) to (4);
\draw  (3) to (5);
\draw  (4) to (6);
\draw  (4) to (7);
\draw  (5) to (8);
\draw  (5) to (9);
\draw  (6) to (7);
\draw  (6) to (8);
\draw  (8) to (9);
\draw  (9) to (7);
\endtikzpicture 
\caption{The graph $F_{3}$.}\label{PD3}
\end{figure}
\end{example}

\begin{theorem}
\label{PDClassWellCovered}
The graph $F_{m}$ is well-covered for $m \geq 2$.
\end{theorem}

\begin{proof}
Note that we can partition the vertex set of $G=F_{m}$ 
into $X_{1},X_2$ and $Z$. Further,
$G[X_1]$, $G[X_2]$ and $G[Z]$ are all complete graphs. 
Hence, any maximal independent set will have cardinality $3$ or fewer. 
Let $H$ be an independent set of $G$. 
Suppose $Z \cap H =\emptyset$. Then $H \cup \{z_3\}$ is an independent set since $z_3$ 
is only adjacent to vertices in $Z$. Thus $Z \cap H \neq \emptyset$. 
Suppose $X_1 \cap H =\emptyset$. 
If $y_2$ is in $H$ or $H\cap X_2 = \emptyset$, 
let $x=x_1$. Otherwise let $x=x_{2k-1}$ if $x_{2k}$ is a vertex in $H$.
Then $H \cup \{x\}$ is an independent set. Thus $X_1 \cap H \neq \emptyset$
and by symmetry $X_2 \cap H \neq \emptyset$.

Therefore all maximal independent sets of $G$ must have cardinality $3$, so $F_{m}$ is well-covered.
\end{proof}

\begin{lemma}
\label{VDTwoAttached}
Given $m\geq 2$, if $G=F_m$, then $G[X \cup Y]$ is pure vertex decomposable.
\end{lemma}

\begin{proof}
Since $G[X \cup Y]$ is a clique-whiskered graph, it  
is pure vertex decomposable by \cite[Theorem~3.3]{CN}.
\end{proof}

\begin{lemma}
\label{VDOneAttached}
Given $m\geq 2$, and $G=F_m$. Let $S=X \cup \{y_1\}$. Then
$G[S]$ 
is pure vertex decomposable. 
\end{lemma}

\begin{proof}
Let $H=G[S]$.
Note that $y_1$ is a simplicial vertex of $H$.
Let $x$ be a vertex adjacent to $y_1$.
The graph $H \setminus N_{H}[x]$ is a single isolated vertex and hence is pure vertex decomposable. 

Note that $H$ is well-covered with $\alpha(H)=2$.  
Thus $H \setminus x$ is well-covered by Lemma \ref{wellcoveredsubgraph}. 
Using Lemma~\ref{wellcoveredsubgraph} we can 
continue to remove vertices adjacent to $y_1$ while maintaining a well-covered graph
until we obtain the graph with isolated vertex $y_1$ and complete graph on vertex set $X_2\setminus y_2$.
This resultant graph is a union of two complete graphs and hence is pure vertex decomposable by 
Lemma~\ref{union}. Therefore $H \setminus x$ is pure vertex decomposable.
Since $H \setminus N_{H}[x]$ is an isolated vertex, it is pure vertex decomposable.
Therefore $x$ is a shedding vertex of $H$  and $H$ is pure vertex decomposable.  
\end{proof}

Given $\alpha=\alpha(G)$, define $i_r$ to be the number of
independent sets of $G$ of cardinality $r$ for $1\leq r\leq \alpha$
with $i_0=1$.
Define the $h$-vector
$h_G=(h_0,h_1,\ldots,h_{\alpha})$ by 
\[
h_k = \sum_{r = 0}^k (-1)^{k-r}\binom{\alpha -r}{k-r}i_{r} \, .
\]
As noted  in \cite[Theorem 5.4.8]{V2001}, 
if a graph is Cohen-Macaulay, then the $h$-vector is a non-negative 
vector.\footnote{Note that the $f$-vector $(f_0,f_1,\ldots,f_{\alpha-1})$ 
described in \cite{V2001} is $(i_1,i_2,\ldots, i_{\alpha})$
with $f_{-1}=1$. } Since every pure vertex decomposable graph is Cohen-Macaulay, we have
the following restatement 
which we will use to 
limit the cardinality of Shed($F_m$). 

\begin{lemma}[{\cite[Theorem 5.4.8]{V2001}}]\label{hvec}
If $G$ is a pure vertex decomposable graph, then $h_G$ is a non-negative vector.
\end{lemma}

\begin{theorem}\label{pmfamily}
For all $m\geq 2$, $F_{m}$ is pure vertex decomposable and ${\rm Shed}(F_{m})= \{z_{1},z_{2}\}$.
\end{theorem}

\begin{proof}
We first show that if $v\not\in \{z_1,z_2\}$ then $F_m \setminus v$ is not
pure vertex decomposable. 

Suppose that $v \in X$. By the symmetry of the graph, 
we can assume $v=x_1$. Then
$\{ y_1,y_2,z_3\}$ and $\{x_2,z_1\}$ are maximal independent sets
of different cardinality in $F_m \setminus v$. Thus $F_m \setminus v$ is not
well-covered and hence not pure vertex decomposable.

Next we consider a vertex in $v\in Y$. By symmetry, assume $v=y_1$. 
We will show that $F_{m} \setminus v$ is not vertex 
decomposable by showing that its $h$-vector has a negative entry.
We first calculate the number $i_r$ of independent sets of cardinality $r$ in $F_{m} \setminus v$,
for $1\leq r \leq \alpha$.
Note that $\alpha(F_{m} \setminus v)=3$. 
There are $2m+4$ vertices in $F_{m} \setminus v$ so $i_1=2m+4$.
An independent set of cardinality 2 can be of the form $\{y_2,x_i\}$, $\{y_2,z\}$
$\{z,x_i\}$ or $\{x_i,x_j\}$ for some $x_i,x_j\in X$ and $z\in Z$. There are
$m$, $2$, $6m$ and $m$ such different independent sets respectively. 
Thus $i_2=8m+2$.
An independent set of cardinality 3 must have one vertex in $Z$, one in $X_2$ and one in $X_1\setminus y_1$
since these sets partition the vertex set, and induce complete subgraphs, of $F_m \setminus v$ .
There are $m$ maximal independent sets containing $z_2$ and
for each $z\in Z \setminus z_2$, there are $2m$ maximal independent sets containing $z$.
Thus $i_3=5m$.
Therefore 
$(i_0,i_1,i_2,i_3)=(1,2m+4,8m+2,5m)$.
But this implies that the $h$-vector has 
$h_3=1-m$.
Hence $h_3 < 0$ for $m>1$ and by Lemma~\ref{hvec}, $F_{m}\setminus v$  is not pure vertex decomposable.
Thus no vertex in $Y$ can be a shedding vertex of $F_m$ if $F_m$ is pure vertex decomposable.

Since $\{z_1, x_1,x_2\}$ and $\{y_1,y_2\}$ are maximal independent sets
with different cardinalities in $F_m \setminus z_3$, $F_m \setminus z_3$ is not
well-covered and hence not pure vertex decomposable. 

Therefore, if $F_m\setminus v$ is pure vertex decomposable, then $v\in \{z_1,z_2\}$.

Now suppose $v=z_1$. We claim that $F_m \setminus v$ is pure vertex decomposable.
The graph $F_{m} \setminus N_{F_{m}}[z_{1}]$ 
is the graph $G[S]$ described in Lemma \ref{VDOneAttached} and so it is pure vertex decomposable
and hence well-covered.

Next we claim that the graph $G=F_{m} \setminus z_{1}$ is well-covered.
We can partition the vertices of $G$ into the sets 
$Z \setminus z_{1} \cup X_1 \cup X_2$. 
Since each part in the vertex partition induces a complete graph, 
we can construct an independent set of cardinality at most $3$. Thus $\alpha(G)\leq 3$.
Using an argument similar to Lemma~\ref{PDClassWellCovered}, one can 
show that every maximal independent set of $G$ is of cardinality 3
and hence $G=F_{m} \setminus z_{1}$ is well-covered.

We show that $G=F_{m} \setminus z_{1}$ 
is pure vertex decomposable by showing that $z_{2}$
is a shedding vertex of $G$.
First $G
\setminus N_{F_{m}}[z_{2}] = F_{m} \setminus  N_{F_{m}}[z_{2}]$ since $z_{1}$ is adjacent to $z_{2}$, 
and $F_{m} \setminus  N_{F_{m}}[z_{2}]$ is isomorphic to 
$F_{m} \setminus  N_{F_{m}}[z_{1}]$. Thus $G\setminus N_{F_{m}}[z_{2}]$
 is pure vertex decomposable. 
Next, $G\setminus z_2=F_{m} \setminus \{z_{1},z_{2}\}$ 
has an isolated vertex $z_{3}$ and a component described
in Lemma~\ref{VDTwoAttached} and so is pure vertex decomposable by Lemma~\ref{union}.

Therefore $F_{m} \setminus N_{F_{m}}[z_{1}]$ and $F_{m} \setminus z_{1}$ 
are well-covered, so $F_{m}$ is pure vertex decomposable
and it follows that $z_{1}$ (and $z_2$ by symmetry) are shedding vertices
of $F_m$.
\end{proof}

\begin{corollary}\label{corP}
For all $m \geq 2$,  ${\rm Shed}(F_{m})$ is not a dominating set.
\end{corollary}

\begin{proof}
Since each vertex in $X$ is not adjacent to a shedding vertex of $F_{m}$, 
${\rm Shed}(F_{m})$ is not a dominating set.
\end{proof}

\subsection{Construction 3}

We finish this section by describing another family of pure vertex decomposable
graphs whose set of shedding vertices fails to be a dominating set.
Unlike the previous constructions, for the sake of brevity, we only sketch out the details 
of the proof.

Fix an integer $n \geq 1$.  
Let 
\begin{eqnarray*}
X&=&\{x_{1,1},x_{1,2}\}\cup \{x_{2,1},x_{2,2}\} \cup \ldots \cup \{ x_{n,1},x_{n,2}\}, \\
Y&=& \{y_{1,1},y_{1,2},y_{1,3}\}\cup\{y_{2,1},y_{2,2},y_{2,3}\}\cup 
\ldots \cup \{y_{n,1},y_{n,2},y_{n,3}\}, \mbox{\quad {\rm{and}}}\\
Z&=&\{z_{1,1},z_{1,2},z_{1,3}\}\cup\ldots \cup\{z_{n,1},z_{n,2},z_{n,3}\}.
\end{eqnarray*}
We define the graph $L_n$ to be the 
graph on $8n+1$ vertices $V=X \cup Y\cup Z \cup \{w\}$.
with the edge set given by the following conditions:
\begin{enumerate}
\item[$(i)$] for each $i=1,\ldots,n$, the induced graph on $\{x_{i,1},x_{i,2},y_{i,1},y_{i,2},y_{i,3}\}$
is a $5$-cycle with edges $\{y_{i,1},y_{i,2}\},\{y_{i,2},y_{i,3}\},\{y_{1,3},x_{i,2}\},\{x_{i,2},x_{i,1}\},
\{x_{i,1},y_{i,1}\}$;
\item[$(ii)$] $\{z_{i,1},y_{i,1}\},\{z_{i,2},y_{i,2}\}$, and $\{z_{i,3},y_{i,3}\}$ are 
edges for $i=1,\ldots,n$, forming a matching between $Y$ and $Z$; and
\item[$(iii)$] the induced graph on
$Z \cup \{w\}$ is
the complete graph $K_{3n+1}$.
\end{enumerate}

\begin{example}
The graph $L_1$ is given in Figure \ref{EL1} and $L_2$ in Figure \ref{L2}.

\begin{figure}[ht]
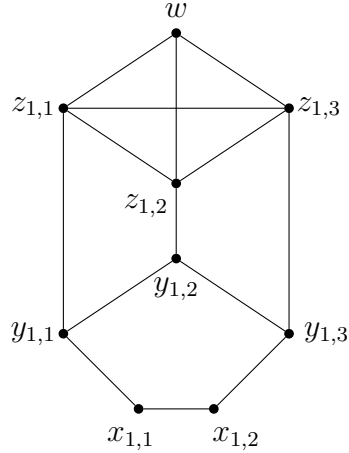

\tikzpicture

\node (x11) at (-3,-3)[place]{};
\node (x11b) at (-3.1, -3.4){$x_{1,1}$};
\node (x12) at (-2,-3)[place]{};
\node (x12b) at (-1.7, -3.4){$x_{1,2}$};
\node (y11) at (-4,-2)[place]{};
\node (y11b) at (-4.4, -2){$y_{1,1}$};
\node (y12) at (-2.5,-1)[place]{};
\node (y12b) at (-2.5, -1.4){$y_{1,2}$};
\node (y13) at (-1,-2)[place]{};
\node (y13b) at (-0.5, -2){$y_{1,3}$};

\draw (x11) to (y11);
\draw (x11) to (x12);
\draw (x12) to (y13);
\draw (y13) to (y12);
\draw (y12) to (y11);

\node (z11) at (-4, 1)[place]{};
\node (z11b) at (-4.4, 1){$z_{1,1}$};
\node (z12) at (-2.5, 0)[place]{};
\node (z12b) at (-2.9, -0.3){$z_{1,2}$};
\node (z13) at (-1, 1)[place]{};
\node (z13b) at (-0.6, 1){$z_{1,3}$};

\draw (y11) to (z11);
\draw (y12) to (z12);
\draw (y13) to (z13);

\node (w) at (-2.5,2)[place]{};
\node (wb) at (-2.5,2.3){$w$};

\draw (z11) to (z12);
\draw (z11) to (z13);
\draw (z11) to (w);

\draw (z12) to (z13);
\draw (z12) to (w);

\draw (z13) to (w);
\endtikzpicture 
\caption{The graph $L_1$.}\label{EL1}
\end{figure}

\begin{figure}[ht]
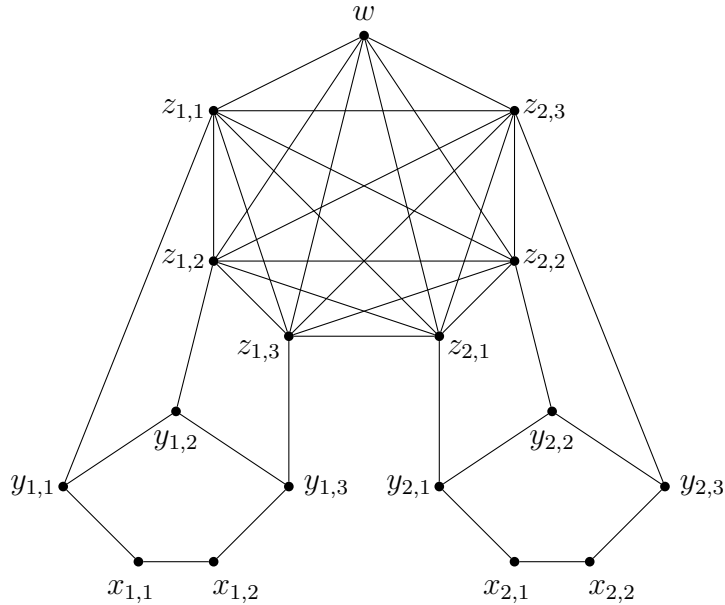

\tikzpicture

\node (x11) at (-3,-3)[place]{};
\node (x11b) at (-3.1, -3.4){$x_{1,1}$};
\node (x12) at (-2,-3)[place]{};
\node (x12b) at (-1.7, -3.4){$x_{1,2}$};
\node (y11) at (-4,-2)[place]{};
\node (y11b) at (-4.4, -2){$y_{1,1}$};
\node (y12) at (-2.5,-1)[place]{};
\node (y12b) at (-2.5, -1.4){$y_{1,2}$};
\node (y13) at (-1,-2)[place]{};
\node (y13b) at (-0.5, -2){$y_{1,3}$};

\draw (x11) to (y11);
\draw (x11) to (x12);
\draw (x12) to (y13);
\draw (y13) to (y12);
\draw (y12) to (y11);

\node (x21) at (2,-3)[place]{};
\node (x21b) at (1.9, -3.4){$x_{2,1}$};
\node (x22) at (3,-3)[place]{};
\node (x22b) at (3.3, -3.4){$x_{2,2}$};
\node (y21) at (1,-2)[place]{};
\node (y21b) at (0.6, -2){$y_{2,1}$};
\node (y22) at (2.5,-1)[place]{};
\node (y22b) at (2.5, -1.4){$y_{2,2}$};
\node (y23) at (4,-2)[place]{};
\node (y23b) at (4.5, -2){$y_{2,3}$};

\draw (x21) to (y21);
\draw (x21) to (x22);
\draw (x22) to (y23);
\draw (y23) to (y22);
\draw (y22) to (y21);

\node (z11) at (-2, 3)[place]{};
\node (z11b) at (-2.4, 3){$z_{1,1}$};
\node (z12) at (-2, 1)[place]{};
\node (z12b) at (-2.4, 1){$z_{1,2}$};
\node (z13) at (-1, 0)[place]{};
\node (z13b) at (-1.4, -0.2){$z_{1,3}$};

\draw (y11) to (z11);
\draw (y12) to (z12);
\draw (y13) to (z13);

\node (z23) at (2, 3)[place]{};
\node (z23b) at (2.4, 3){$z_{2,3}$};
\node (z22) at (2, 1)[place]{};
\node (z22b) at (2.4, 1){$z_{2,2}$};
\node (z21) at (1, 0)[place]{};
\node (z21b) at (1.4, -0.2){$z_{2,1}$};

\draw (y21) to (z21);
\draw (y22) to (z22);
\draw (y23) to (z23);

\node (w) at (0,4)[place]{};
\node (wb) at (0,4.3){$w$};

\draw (z11) to (z12);
\draw (z11) to (z13);
\draw (z11) to (z21);
\draw (z11) to (z22);
\draw (z11) to (z23);
\draw (z11) to (w);

\draw (z12) to (z13);
\draw (z12) to (z21);
\draw (z12) to (z22);
\draw (z12) to (z23);
\draw (z12) to (w);

\draw (z13) to (z21);
\draw (z13) to (z22);
\draw (z13) to (z23);
\draw (z13) to (w);

\draw (z21) to (z22);
\draw (z21) to (z23);
\draw (z21) to (w);

\draw (z22) to (z23);
\draw (z22) to (w);

\draw (z23) to (w);

\endtikzpicture 
\caption{The graph $L_2$.}\label{L2}
\end{figure}

\end{example}

We then have the following theorem, whose proof we only sketch.

\begin{theorem}
\label{EClassWellCovered}
For any integer $n \geq 1$,  $L_n$ is pure vertex decomposable,
but ${\rm Shed}(L_n)$ is not a dominating set.
\end{theorem}

\begin{proof}
Suppose $G = L_n$.
To show that $G$ is well-covered, 
show that every maximal independent set has  cardinality $2n+1$.  

To show that 
$G$ is pure vertex decomposable, one can do induction on $n$. 
For $n =1$, one can show that $G$ is pure vertex decomposable directly.  For $n \geq 2$,
let $G_1 = G \setminus z_{n,1}$, $G_2 = G_1 \setminus z_{n,2}$ and $G_3 = G_2 \setminus z_{n,3}$.
Furthermore, let $N_1 = G \setminus N[z_{n,1}]$, $N_2 = G_1 \setminus N[z_{n,2}]$,
and $N_3 = G_2 \setminus N[z_{n,3}]$.

First show that all of the graphs $G_1,G_2,G_3,N_1,N_2$ and $N_3$ 
are well-covered.
We note that $N_1, N_2$, and $ N_3$ are isomorphic because
of the symmetry of the graph, and each graph 
consist of $n$ connected components, where
$(n-1)$ of these components are five cycles, and the last is the path of four vertices.
All of these components are vertex decomposable, thus so is $N_i$.  The graph $G_3$
consists of two components, $L_{n-1}$ and a five cycle.  By induction, these graphs
are pure vertex decomposable.  Using these facts, we can show that $G$ is pure vertex decomposable.

Note to show that $G$ is vertex decomposable, we show that $Z \subseteq {\rm Shed}(G)$.
The next step of the proof is to show
that $X \cap {\rm Shed}(G) = \emptyset$ and $Y \cap {\rm Shed}(G) =
\emptyset$
by showing that if we remove any vertex $v \in X \cup Y$, then $G \setminus v$ is not
well-covered.  This shows that ${\rm Shed}(G)$ is not a dominating set since the vertices
of 
$X$ are only adjacent to vertices in $Y$, but no vertex of $Y$ belongs to ${\rm Shed}(G)$.
\end{proof}


\section{Graph expansions}\label{extendgraph}

In this section we briefly describe a way to extend any pure vertex decomposable
graph whose shedding set is not a dominating set, to build a larger graph with 
the same property by adding one vertex at a time. The technique involves \emph{duplicating}
a vertex in the shedding set.

\begin{theorem}\label{extend}
Suppose $G$ is a pure vertex decomposable graph and ${\rm Shed}(G)$ 
is not a dominating set. For
any $x\in {\rm Shed}(G)$, let $H$ be the graph with $V(H)=V(G)\cup \{x'\}$
and $E(H)=E(G)\cup \{ \{x',y\} ~|~ y\in N[x] \}$. Then $H$ is pure vertex 
decomposable and ${\rm Shed}(H)$ is not a dominating set.
\end{theorem}

To prove Theorem~\ref{extend}, we use a result of \cite{MK}. First we 
define a graph expansion. Let $G$ be a graph on the vertex set $\{x_1,\ldots,x_n\}$
and let $(s_1,\ldots,s_n)$ be an  
$n$-tuple of positive integers.  The {\it graph expansion of $G$},
denoted $G^{(s_1,\ldots,s_n)}$, is the graph on the vertex set
$$\{x_{1,1},\ldots,x_{1,s_1}\}\cup \{x_{2,1},\ldots,
x_{2,s_2}\}\cup \ldots \cup\{x_{n,1},\ldots,x_{n,s_n}\} $$
with edge set 
$\{\{x_{i,j},x_{k,l}\} ~|~ \{x_i,x_k\} \in E(G) ~~\mbox{or}
~~i = k\}.$ Moradi and Khosh-Ahang \cite[Theorem~2.7]{MK} showed that 
vertex decomposability is invariant under graph expansion, that is, 
$G$ is vertex decomposable if and only if $G^{(s_1,\ldots,s_n)}$ is vertex 
decomposable.   Note that a similar construction of ``twinning'' can be
found in \cite{W}.   Moradi and Khosh-Ahang's construction
can be viewed as repeated twinning.

\begin{proof}(of Theorem \ref{extend})
Suppose $G$ is a pure vertex decomposable graph with $V = \{x_1,\ldots,x_n\}$ 
and ${\rm Shed}(G)$ is not a dominating set of $G$. 
Suppose $x \in {\rm Shed}(G)$ and $H$ is a graph with $V(H)=V \cup \{x'\}$
and $E(H)=E(G)\cup \{ \{x',y\} ~|~ y\in N[x] \}$. Without loss of generality, 
assume $x=x_1$. Note that $H=G^{(2,1,\ldots,1)}$
and hence $H$ is pure vertex decomposable since vertex 
decomposability is preserved
under graph expansion, as well as the well-covered property. 

Observe that $x,x'\in {\rm Shed}(H)$, since 
$H\setminus x$ and $H\setminus x'$ are both isomorphic to
$G$ and $G$ is pure vertex decomposable.

Suppose $y\in V$ but $y\not\in {\rm Shed}(G)$.  
We claim $y\not\in {\rm Shed}(H)$. Suppose $y\in {\rm Shed}(H)$. Then $H\setminus y$ is pure vertex decomposable. Note that $H\setminus y$ is a graph expansion of $(H\setminus y)\setminus x'$ and hence $(H\setminus y)\setminus x'$ is 
pure vertex decomposable.
Now, $(H\setminus y)\setminus x'$ is isomorphic to $G\setminus y$, so $G\setminus y$ is pure vertex decomposable. But this contradicts the fact that 
$G\setminus y$ is not pure vertex decomposable if $y\not\in {\rm Shed}(G)$. 
Thus $y \not \in {\rm Shed}(H)$. 

In particular, ${\rm Shed}(H)\setminus \{x'\} \subseteq {\rm Shed}(G)$. 
It follows that ${\rm Shed}(H)$ is
not a dominating set of $H$ since a dominating set of $H$ that includes both $x$ and $x'$ would essentially be a dominating set of $G$ (since having both $x$ and $x'$ in a dominating set is redundant).
\end{proof}

It may be worth noting that it is also possible to 
construct pure vertex decomposable graphs $G$  for which ${\rm Shed}(G)$ is a dominating set 
via graph expansion. As observed in the proof above, the vertex $x$ that gets
duplicated as well as its duplicate $x'$ are both in the set of shedding vertices in the graph
expansion. It follows that if every vertex is duplicated at least once on a 
pure vertex
decomposable graph, the resulting graph will be pure vertex decomposable with every vertex 
in its shedding set.  Consequently, many graph expansions $G$ have ${\rm Shed}(G)$ as a dominating set: 

\begin{theorem}
If $G$ is any pure vertex decomposable graph and $s_i\geq 2$ for $1\leq i\leq n$, then
$G^{(s_1,s_2,\ldots,s_n)}$ is pure 
vertex decomposable and  ${\rm{Shed}}\left( G^{(s_1,s_2,\ldots,s_n)}\right)$
is a dominating set.
\end{theorem}

\section{Exploring Villarreal's conjecture}\label{insight}

This paper was partially motivated by a conjecture
of R. Villarreal \cite{V1990} about Cohen-Macaulay graphs. (Every
Cohen-Macaulay graph is well-covered.) In particular,   
Villarreal \cite{V1990} introduced the notion of an {\it edge ideal}
of $G$, that is, in the polynomial ring $R = k[x_1,\ldots,x_n]$ 
over a field $k$, let $I(G)$ denote the square-free
quadratic monomial ideal
$I(G) = \langle x_ix_j ~|~ \{x_i,x_j\} \in E \rangle.$
A graph $G$ is {\it Cohen-Macaulay} if 
the quotient ring $R/I(G)$ is a Cohen-Macaulay ring, that is,
the depth of $R/I(G)$ equals the Krull dimension of $R/I(G)$.  
The goal of \cite{V1990} was to determine
necessary and sufficient conditions for a graph to
be Cohen-Macaulay. 
Based upon computer experiments on all graphs on six or
fewer vertices, Villarreal 
proposed a two-part conjecture:

\begin{conjecture}[{\cite[Conjectures 1 and 2]{V1990}}]\label{CMconjecture}
Let $G$ be a Cohen-Macaulay graph
and let 
\[S = \{x \in V ~|~ G \setminus x \mbox{ is a Cohen-Macaulay graph}\}.\] 
Then $(i)$ $S \neq \emptyset$, and 
$(ii)$ $S$ is a dominating set of $G$.
\end{conjecture}

Notice that $(ii)$
will not hold if $(i)$ does not hold.
It is known that Conjecture \ref{CMconjecture} $(i)$ is false.  
One example is due to Terai
\cite[Exercise 6.2.24]{V2001}.  However
Terai's example depends upon the characteristic of the field
$k$.  Earl et al.  
\cite{EVMVT} found an
example of a circulant graph $G$ on 16 vertices with the property
that $G$ is Cohen-Macaulay in all characteristics, but
there is no vertex $x$ such that $G \setminus x$ is Cohen-Macaulay.

Although Conjecture \ref{CMconjecture} is false in general, 
Villarreal's work suggests that there may exist some 
nice subset of Cohen-Macaulay graphs for 
which the Conjecture \ref{CMconjecture}
still holds, particularly the subset of Cohen-Macaulay graphs for which
$S\neq \emptyset$. 
Since pure vertex decomposable
graphs are Cohen-Macaulay 
(pure vertex decomposable complexes 
are shellable complexes \cite[Corollary 2.9]{PB}, 
and shellable complexes are Cohen-Macaulay) and since ${\rm Shed}(G) \neq \emptyset$ for pure vertex decomposable
graphs $G$, we thought that, 
as a variation of 
Conjecture \ref{CMconjecture}, 
 it would be reasonable to question 
if ${\rm Shed}(G)$ is a dominating set of $G$ when $G$ is pure vertex decomposable. 
The number of positive answers to our question, 
as observed
in Theorem~\ref{summary}, initially 
suggested a positive answer for all pure vertex decomposable graphs.
However, our examples in Section~\ref{newgraphs} demonstrate that the answer
is not positive in general.

We conclude with some computational observations.
 We used {\it Macaulay2}
\cite{GS}  and the packages {\tt EdgeIdeals} \cite{FHVT}, 
{\tt Nauty} \cite{C2011}, and
{\tt SimplicialDecomposability} \cite{C2010} 
for our computations.
For all connected graphs $G$ on 10 or fewer vertices, we checked whether the graph
was $(a)$ well-covered, $(b)$ Cohen-Macaulay, $(c)$ pure vertex decomposable, and
$(d)$ if the graph was pure vertex decomposable, whether ${\rm Shed}(G)$ is a dominating
set.
Table \ref{table}
summarizes our findings.  The first column is the number of
vertices, while the second column is the number of connected graphs on $n$
vertices, and the third column is the number of well-covered graphs on $n$ 
vertices.  The second column is sequence
A001349 in the OEIS, and the third column is sequence A2226525 in the OEIS \cite{OEIS}.

As part of this computer experiment, we 
counted the number of Cohen-Macaulay graphs.  The fourth and fifth
columns of Table~\ref{table} count the number of Cohen-Macaulay graphs, respectively, the
number of pure vertex decomposable graphs.  Our computations imply 
the following result:

\begin{observation}\label{VDfew}
Let $G$ be a graph with $10$ or fewer  
vertices.  Then $G$ is Cohen-Macaulay if and only
if $G$ is pure vertex decomposable.
\end{observation}

\noindent
It is not true that all graphs that are Cohen-Macaulay are pure vertex decomposable (see, e.g.,
\cite{EVMVT} for a graph on 16 vertices that is Cohen-Macaulay, but not pure vertex decomposable).
However, we currently do not know the smallest such example.  Our computations
reveal that the minimal such example has at least 11 vertices.

The last column counts the number of pure vertex decomposable graphs $G$ for which
${\rm Shed}(G)$ is not a dominating set. 
Among the 17 graphs $G$ on 9 vertices for which 
${\rm Shed}(G)$ is not a dominating set, 
we found that the graph $F_2$ (see Figure~\ref{PD2}) 
has the least number of edges.

\vspace{.25cm}
\begin{table}[h!]
\begin{center}
\begin{tabular}{|c|c|c|c|c|c|}
\hline 
Vertices & Connected & Well- &Cohen- & Pure Vertex & ${\rm Shed}(G)$ not \\
& Graphs &Covered &Macaulay & Decomposable&  Dominating \\
\hline 
1 & 1 & 1& 1 & 1 & 0 \\ 
\hline 
2 & 1 & 1& 1 & 1 & 0 \\ 
\hline 
3 & 2 & 1& 1 & 1 & 0 \\ 
\hline 
4 & 6 & 3&2 & 2 & 0 \\ 
\hline 
5 & 21 & 6& 5 & 5 & 0 \\ 
\hline 
6 & 112 & 27&20 & 20 & 0 \\ 
\hline 
7 & 853 & 108& 82 & 82 & 0 \\ 
\hline 
8 & 11117 & 788& 565 & 565 & 0 \\ 
\hline 
9 & 261080 & 9035& 5688 & 5688 & 17 \\ 
\hline 
10 & 11716571 & 196928 &102039  & 102039 & 942\\ 
\hline 
\end{tabular} 
\caption{Number of well-covered, Cohen-Macaulay, and pure vertex decomposable graphs}\label{table}
\end{center}
\end{table}

\begin{observation}\label{obs}
Conjecture \ref{CMconjecture} is true for all Cohen-Macaulay
graphs on eight or fewer vertices.  The graph $F_2$ on
nine vertices and 13 edges is the minimal 
counterexample to
Conjecture \ref{CMconjecture}. 
\end{observation}

\noindent 
{\it Rationale for Observation~\ref{obs}:}
Let $G$ be any Cohen-Macaulay graph
and let 
\[S = \{x \in V~|~ G \setminus x ~~\mbox{is a Cohen-Macaulay graph}\}.\]
If $G$ is also pure vertex decomposable and if $x \in {\rm Shed}(G)$, then
$G \setminus x$ is pure vertex decomposable, so $G \setminus x$ 
is Cohen-Macaulay.  So, we always have ${\rm Shed}(G) \subseteq S$.

If $G$ is a Cohen-Macaulay graph on eight or fewer vertices,
it is also pure vertex decomposable by Remark~\ref{VDfew}.  Also,
our computations imply that ${\rm Shed}(G)$ is a dominating set for all
such graphs and hence $S$ is also a dominating set. 

In our proof Theorem~\ref{pmfamily}, we showed that $F_2$ is a pure vertex
decomposable graph.  Furthermore, for
every vertex $x \in V(F_2) \setminus {\rm Shed}(F_2)$, the graph
$F_2 \setminus x$ is either not well-covered (and thus not
Cohen-Macaulay) or not Cohen-Macaulay.  
So, ${\rm Shed}(F_2) = S$, and thus $F_2$ is a counterexample to Conjecture~\ref{CMconjecture} 
by Corollary~\ref{corP} since the shedding set is never
dominating.  The minimality in our statement follows
via our computations: of the counterexamples on nine vertices, $F_2$ has the least number of edges.
\qed 

\noindent
\textbf{Acknowledgements.} We would like to thank 
Nuno Alves for pointing out a gap in our original proof of Lemma 3.5, and
the referees for helpful suggestions and improvements.
Research supported in part by NSERC. This work was made possible by the facilities of the Shared Hierarchical 
Academic Research Computing Network (SHARCNET: {\tt www.sharcnet.ca}) and Compute/Calcul Canada. The third author
wrote part of this paper at the Vietnam Institute for Advanced Study 
in Mathematics (VIASM), and would like to thank the institute for their
hospitality and support during his stay.


\end{document}